\documentclass[a4paper,10pt,reqno]{amsart}

\usepackage{amsmath,graphics}
\usepackage{amssymb}
\usepackage{amsfonts}
\usepackage{latexsym}
\usepackage{eucal}
\usepackage[dvips]{graphicx}
\usepackage{parskip}

\theoremstyle{plain}
\newtheorem{thm}{Theorem}[section]
\newtheorem{propo}[thm]{Proposition}
\newtheorem{lem}[thm]{Lemma}
\newtheorem{cor}[thm]{Corollary}

\theoremstyle{definition}
\newtheorem{defi}[thm]{Definition}
\newtheorem{rem}[thm]{Remark}

\renewcommand{\Re}{{\rm Re}}

\newcommand{\R}{\mathbb{R}}
\newcommand{\C}{\mathbb{C}}
\newcommand{\Z}{\mathbb{Z}}
\newcommand{\N}{\mathbb{N}}

\newcommand{\T}{\mathbb{T}}

\newcommand{\lt}{{\mathcal L}}

\newcommand{\U}{{\mathcal U}}

\newcommand{\Oo}{{\mathcal O}}
\newcommand{\PSH}{\mathrm {PSH}}

\newcommand{\RS}{\hat{\mathbb{C}}}

\newcommand{\abs}[1]{\left | {#1}\right | }
\newcommand{\norm}[2]{\left\| {#1}\right\| _{#2}}

\newcommand{\all}[2]{ \{\, {#1} \, : \, {#2} \, \} }
\newcommand{\cl}[1]{{\rm cl} ( #1 )} % closure of a set

\newcommand{\HR}{H^2(D^\infty_R)}%{\hat{\mathbb{C}}}}
\newcommand{\HRnull}{H^2_0(D^\infty_R)}
\newcommand{\monster}{H^2(D_r)\oplus H^2_0(D^\infty_R)}

\newcommand{\Tr}{{\rm Tr}}

\begin{document}\bibliographystyle{plain}
\title[Ruelle spectrum of analytic expanding circle maps]{Lower bounds for the Ruelle spectrum of analytic expanding circle maps}

\author[O.F.~Bandtlow]{Oscar F.~Bandtlow}
\address{%
Oscar F.~Bandtlow\\
School of Mathematical Sciences\\
Queen Mary University of London\\
London E3 4NS\\
UK.
}
\email{o.bandtlow@qmul.ac.uk}

\author[F.~Naud]{Fr\'ed\'eric Naud}
\address{%
Fr\'ed\'eric Naud\\
Laboratoire de Math\'ematiques d'Avignon\\
Universit\'e d'Avignon, Campus Jean-Henri Fabre, 301 rue de Baruch de Spinoza,\\
84916 Avignon cedex 9,\\
France.
}
\email{frederic.naud@univ-avignon.fr}



\begin{abstract}
 We prove that there exists a dense set of analytic expanding maps of
 the circle for which the Ruelle eigenvalues enjoy exponential lower bounds. 
The proof combines potential theoretic techniques and explicit calculations for the spectrum of expanding Blaschke products. 
\end{abstract}
 
\maketitle

\begin{section}{Introduction and statement}
One of the basic problems of smooth ergodic theory is to investigate
the asymptotic behaviour of mixing systems. In particular, there is
interest in precise quantitative
results on the rate of decay of correlations for smooth observables. 
In his seminal paper \cite{Ruelle4}, Ruelle showed that the long time
asymptotic behaviour 
of {\it analytic} hyperbolic systems can be understood in terms of
the  {\it Ruelle spectrum}, that is,  
the spectrum of certain operators, known as transfer operators in this
context,  
% (the {\it Ruelle spectrum}),
acting on suitable Banach spaces of holomorphic functions. 
Surprisingly, even for the simplest
systems like analytic expanding maps of the circle, very few
quantitative results on the Ruelle spectrum are known so far. 

Let us be more precise. With $\T$ denoting the unit circle in the
complex 
plane $\C$, a map $\tau: \T \rightarrow \T$ is said to be 
{\it analytic expanding} 
if $\tau$ has a holomorphic extension to a neighbourhood
of $\T$ and we have 
$$ \inf_{z\in \T}\abs{\tau'(z)}>1\,.$$ 
The Ruelle spectrum is commonly defined as the spectrum of the 
Perron-Frobenius or transfer operator (on a suitable space of 
holomorphic functions) given locally by 
\begin{equation}
\label{eq:FPOdef}
(\lt_\tau f)(z):=\omega(\tau)\sum_{k=1}^d \phi_k'(z) f(\phi_k(z))\,,
\end{equation}
where $\phi_k$ denotes the $k$-th local inverse branch 
of the covering map $\tau:\T \rightarrow \T$, and 
\begin{equation}
\label{eq:omegadef}
\omega(\tau)= 
\begin{cases}
  +1 & \text{if $\tau$ is orientation preserving;} \\
  -1 & \text{if $\tau$ is orientation reversing.} \\
\end{cases}
\end{equation}
As first demonstrated by 
Ruelle in \cite{Ruelle4}, this transfer operator
has a discrete spectrum of eigenvalues which has an 
intrinsic dynamical meaning and does not depend on the choice of
function space. 
Let $(\lambda_n(\lt_\tau))_{n \in \N}$
denote this eigenvalue sequence, counting algebraic multiplicities and
ordered by decreasing 
modulus, so that 
$$
1=\lambda_1(\lt_\tau)>\vert \lambda_2(\lt_\tau)\vert \geq \vert \lambda_3(\lt_\tau) \vert
\geq \cdots \geq \vert \lambda_n(\lt_\tau) \vert\geq \cdots \geq 0\,.
$$ 

To state the main result, given $0<r<1<R$, we denote by $A_{r,R}$ the
complex annulus defined by
$$A_{r,R}=\all{z \in \C}{r<\vert z\vert <R}\supset \T\,.$$

In this paper we prove the following facts.
\begin{thm}
Let $\tau$ be an analytic expanding map of the circle as defined
above. Then the following holds.  
\begin{enumerate}
\item There exist  constants $c_1,c_2>0$ such that for all $n\in \N$ we have 
$$ \vert \lambda_n(\lt_\tau) \vert \leq c_1 \exp(-c_2n)\,.$$
\item Assume that $\tau$ 
is holomorphic on $A_{r,R}$ for some $0<r<1<R$. Then there exist $r<r_1<1<R_1<R$ such that for all $\eta>0$, one can find
an expanding circle map $\tau_\eta$ 
holomorphic on $A_{r_1,R_1}$ with 
$$\sup_{z\in A_{r_1,R_1}}\vert \tau(z)-\tau_\eta(z)\vert \leq \eta, $$
and such that for all $\epsilon>0$, we have 
$$\limsup_{n\rightarrow \infty} \vert \lambda_n(\lt_{\tau_{\eta}})
\vert 
\exp(n^{1+\epsilon}) >0\,.$$
\end{enumerate}
\end{thm}

The first statement is a standard fact and follows from the original
paper of Ruelle: 
the sequence of eigenvalues enjoys an exponential upper bound. 
The second statement shows that for a dense set of analytic circle maps, the
Ruelle spectrum is 
infinite and has a purely exponential decay: the upper bound is 
optimal. Alternatively, if one sets for $r>0$
\[ 
N_\tau(r):=\#\all{n\in \N}{\abs{\lambda_n(\lt_{\tau})}\geq r}\,,
\] 
then the above theorem says that for all analytic expanding maps
$\tau$ 
we have, as $r\rightarrow 0$,
$$N_\tau(r)=O(\abs{\log(r)})\,,$$ 
while for a dense set of analytic maps $\tau$ we have, for 
all $\epsilon >0$, 
$$N_\tau(r)=\Omega\left( \vert\log(r)\vert^{1-\epsilon} \right)\,;$$
here, the ``omega'' notation $f(r)=\Omega(g(r))$ means that there 
does {\it not} exist $C>0$ such that for all $r>0$ small 
we have $f(r)\leq Cg(r)$.

Notice that this statement cannot hold for all analytic expanding
circle maps. Indeed, the 
popular expanding maps $z\mapsto z^d$ have a trivial Ruelle spectrum 
$\{0,1\}$, see, for example, \cite{BJS2}. 

The paper is organised as follows. In the next section, 
we define a general class of holomorphic maps on annuli that
slightly 
generalises the class of analytic expanding maps of the circle. 
We show that it is
possible to define a Hilbert space of hyper-functions on which the composition
(Koopman) operator has a discrete spectrum of eigenvalues, together
with an exponential 
\emph{a priori} upper bound. In the case of circle maps, 
this discrete spectrum
turns out to be the same as the Ruelle spectrum, 
by a standard duality argument. In the following 
Section~\ref{sec:Blaschke} we revisit the main result
from \cite{BJS2} which gives an explicit expression for the spectra of
transfer operators arising from Blaschke products. After
this, we shall derive a similar expression for the spectra of transfer
operators arising from anti-Blaschke products (defined as the reciprocals of
Blaschke products). Both these results will be used in a critical way
later on in the proof of our main result.
In Section~\ref{sec:potentialth}, 
we recall the necessary potential theoretic
background which is the core 
of the main proof. In Section~\ref{sec:holodeform}
we prove a key lemma on deformations of
circle maps in the space of annulus maps. Finally, in the last
section, 
we gather all the previous facts to give a proof of the main theorem.

We hope that the ideas and techniques used here in a  
one-dimensional setup can serve as a blueprint for future work 
related to the Ruelle spectrum of Anosov maps and flows. 
This will be pursued elsewhere.
\end{section}

\begin{section}{Holomorphically expansive maps of the annulus and 
upper spectral bounds for their transfer operators}

In this section we first define a class of holomorphic maps on an
annulus, termed `holomorphically expansive', which 
mildly generalises the class of analytic expanding circle maps. We
then introduce Hardy-Hilbert spaces 
over disks and annuli and show that on these spaces 
composition operators given by holomorphically expansive maps have
exponentially decaying eigenvalues. For the connection with the Ruelle
eigenvalue sequence, we recall a useful representation for the dual of the
transfer operator arising from analytic expanding circle maps as a 
composition operator described in \cite{BJS2}. 
Similar representations for transfer operators associated
with certain rational maps have been given in \cite{LevinSodin}, where
explicit expressions for the corresponding Fredholm determinants can
also be found. 
% Finally, we use this representation to derive exponential
% upper bounds for the eigenvalue sequence of transfer operators given
% by holomorphically expansive maps. 

We start by fixing notation. For $r>0$ we use 
\begin{equation*}
\mathbb{T}_r=\{z\in\mathbb{C}:|z|=r\}\,,
\end{equation*}
to denote circles centred at $0$, and 
\begin{equation*} 
D_r=\{z\in \mathbb{C}:|z|<r\}\,,\quad 
D^\infty_r=\{z\in \hat{\mathbb{C}}: |z|>r\}\,,\quad 
\mathbb{D}=D_1
\end{equation*}
to denote disks centred at $0$ and $\infty$. 
\begin{defi}
Let $A_{r,R}$ be an annulus and let $\tau:A_{r,R}\to \C$. We shall call 
$\tau$ \emph{holomorphically expansive on $A_{r,R}$}, or simply 
\emph{holomorphically expansive} if the annulus is understood, if 
$\tau$ is holomorphic on the closure of $A_{r,R}$  and we have
$$\tau(A_{r,R})\supset \cl{A_{r,R}}.$$
Here, $\cl{A}$ is the closure of a
subset $A$ of the Riemann sphere $\hat{\C}=\C\cup \{\infty\}$.
\end{defi}
\begin{rem} Suppose that $\tau$ is holomorphic on the closed annulus $\cl{A_{r,R}}$. 
Since any  non-constant holomorphic map is open, we must have
$\partial \tau(A_{r,R})= \tau(\partial A_{r,R})$ and it follows that $\tau$ is holomorphically expansive on 
$A_{r,R}$ if and only if either of the following two alternatives hold:
\begin{itemize}
\item[(A1)] $ \tau(\mathbb{T}_r)\subset D_r$ and $\tau(\mathbb{T}_R)\subset D_R^\infty$;
\item[(A2)]  $\tau(\mathbb{T}_r)\subset D_R^\infty$ and $\tau(\mathbb{T}_R)\subset D_r$.
\end{itemize}
If (A1) holds we shall call $\tau$ \emph{orientation preserving}, while if (A2) holds we shall call $\tau$ \emph{orientation reversing}. 
\end{rem}
% Notice that it follows from that 
% definition\footnote{Since any non-constant holomorphic map is open, we must have
% $\partial \tau(A_{r,R})= \tau(\partial A_{r,R})$.}
% that either $\tau(\mathbb{T}_r)\subset D_r$ and $\tau(\mathbb{T}_R)\subset D_R^\infty$
% or $\tau(\mathbb{T}_r)\subset D_R^\infty$ and $\tau(\mathbb{T}_R)\subset D_r$.
It is not difficult to see that every analytic expanding circle map is 
holomorphically expansive on all sufficiently small annuli containing
the unit circle 
(see, for example, \cite[Lemma~2.2]{SBJ_nonlinearity}; 
this is also a special case of Lemma~\ref{lem:analyticdef}, to be proved
later), but not the other way round. As we shall see
shortly, with every holomorphically expansive $\tau$ it is possible to 
associate an operator with discrete spectrum, which, in case $\tau$ leaves the unit
circle invariant, coincides with the Ruelle spectrum of $\tau$, that
is, the spectrum of $\lt_\tau$. 

In order to make this connection more precise, we first need to
introduce appropriate spaces of holomorphic functions on which these
operators act. 
The positively oriented boundary of a disk or an annulus 
will be denoted by $\partial_+$.
For $U$ an open subset of $\RS$ we write  
$\operatorname{Hol}(U)$ to denote the space of holomorphic functions on $U$.

% Hardy-Hilbert spaces on disks and annuli will provide a convenient
% setting for our analysis. We briefly recall their properties in the following. 

%%%%%%% Function spaces %%%%%%%%%%%%%%%%%
\begin{defi}\label{defn:HardyHilbert}
%Let $D_r$ and $A_{r,R}$ be as above and 
For $\rho>0$ and $f: \mathbb{T}_\rho \to \mathbb{C}$ 
write
\[
   M_\rho(f)=\frac{1}{2\pi}\int_{0}^{2\pi}|f(\rho e^{i\theta})|^2\, 
  d\theta \,.
\]
Then 
\begin{align*}
H^2(D_r)      & = \all{f\in \operatorname{Hol}(D_r)}{
              \sup_{\rho < r} M_\rho(f)<\infty} \\ 
H^2(A_{r,R})   & = \all{f\in \operatorname{Hol}(A_{r,R})}{
                 \sup_{\rho > r} M_\rho(f)+ 
                 \sup_{\rho < R} M_\rho(f) <\infty} \\
H^2(D^\infty_R) & = \all{f\in \operatorname{Hol}(D^\infty_r)}{\sup_{\rho > R} M_\rho(f) <\infty}
\end{align*} 
are called the \emph{Hardy-Hilbert spaces} on $D_r$, $A_{r,R}$, 
and $D^\infty_R$, respectively. 

The subspace  of $\HR$ consisting of functions vanishing at infinity
will be denoted by $\HRnull$. 
\end{defi}

The classic text \cite{Duren} gives a 
comprehensive account of Hardy spaces over general domains. 
Hardy spaces on the unit disk are discussed in considerable detail 
in \cite[Chapter 17]{Rudin2}), while a good
reference for Hardy spaces on annuli is \cite{Sarason1965}. 

We briefly mention a number of results which will be useful in what
follows. Any function in $H^2(U)$, where $U$ is a disk or an annulus, can be extended 
to the boundary in the following sense. 
For any $f\in H^2(D_r)$ there is 
$f^*\in L^2(\mathbb{T}_r)=L^2 ( \mathbb{T}_r,d\theta/2\pi)$ 
the usual Hilbert space of square-integrable functions 
with respect to normalized one-dimensional Lebesgue measure
on $\mathbb{T}_r$, such that
\begin{equation*}%\label{eq:boundaryFunc}
 \lim_{\rho\uparrow r} f(\rho e^{i\theta}) = f^*(re^{i\theta}) \quad \text{for a.e.~$\theta$},
\end{equation*}
and analogously for $f\in \HR$. Similarly, for $f\in H^2(A_{r,R})$ there are
$f^*_1\in L^2(\mathbb{T}_r)$ and $f^*_2\in L^2(\mathbb{T}_R)$,
with 
$$\lim_{\rho\downarrow r} f(\rho e^{i\theta}) = f_1^*(re^{i\theta}) 
\text{ and } 
\lim_{\rho\uparrow R} f(\rho e^{i\theta}) = f_2^*(Re^{i\theta})
\text{ for a.e.~$\theta$}\,.
$$  
 The above terminology is justified since the spaces $H^2(U)$ 
turn out to be Hilbert spaces with inner products 
\[
  (f,g)_{H^2(D_r)} = 
   \frac{1}{2\pi}
   \int_{0}^{2\pi}f^*(re^{i\theta})\overline{g^*(re^{i\theta})}\,d\theta
\]
and
\[
  (f,g)_{H^2(A_{r,R})} = 
    \frac{1}{2\pi}\int_{0}^{2\pi}
       f_1^*(Re^{i\theta})\overline{g_1^*(Re^{i\theta})}\,d\theta +
    \frac{1}{2\pi}
    \int_{0}^{2\pi}
       f^*_2(re^{i\theta})\overline{g^*_2(re^{i\theta})}\,d\theta\,. 
\]
Similarly, for $\HR$.

\begin{rem}
In the following we shall write $f(z)$ instead of $f^*(z)$ for $z$ on the
boundary of the domain, if this does not lead to confusion.  
\end{rem}

\begin{rem}\label{rem:Basis}
For $\rho>0$ and $n\in \Z$ let 
\[ e^{(\rho)}_n(z)=\frac{z^n}{\rho^n}\,.\]
It is not difficult to see that 
$\all{e_n^{(\rho)}}{n\in \N_0}$ is an orthonormal basis for $H^2(D_\rho)$  and that 
$\all{e_{-n}^{(\rho)}}{n\in \N}$ is an orthonormal basis for $H_0^2(D_\rho ^\infty)$. 
\end{rem}

For later use, we note the following simple consequence of the above
remark. 

\begin{lem} 
  \label{lem:pevalcont}
For any $f\in H^2(D_r)$ and any $z\in D_r$ we have
\begin{equation}
  \label{eq:pevalDr}
  \abs{f(z)}\leq \frac{r}{\sqrt{r^2-\abs{z}^2}} \norm{f}{H^2(D_r)}\,.
\end{equation}
Similarly, for any $f\in H^2_0(D_R^\infty)$ and any $z\in D_R^\infty$ we have
\begin{equation}
  \label{eq:pevalDRinf}
  \abs{f(z)}\leq \frac{R}{\sqrt{\abs{z}^2-R^2}} \norm{f}{H^2_0(D_R^\infty)}\,.
\end{equation}
\end{lem}

\begin{proof}
We shall only prove (\ref{eq:pevalDRinf}). The proof of
(\ref{eq:pevalDr}) is similar. Let $f\in H^2_0(D_R^\infty)$. Then $f$
has an orthonormal expansion of the form 
\[ f=\sum_{n=1}^\infty f_n e_{-n}^{(R)}\,.\]
The Cauchy-Schwarz inequality now implies that for $z\in
D_R^\infty$ we have 
\begin{equation*}
  \abs{f(z)}^2 \leq 
   \sum_{n=1}^\infty \abs{f_n}^2 \sum_{n=1}^\infty
   {\frac{\abs{R}^{2n}}{\abs{z}^{2n}}}
   =\norm{f}{H^2_0(D_R^\infty)}^2 \frac{R^2}{\abs{z}^2-R^2}  
\end{equation*}
and the assertion follows. 
\end{proof}

We now recall a number of facts from \cite{BJS2} 
which will be crucial for what is to 
follow. We start with the simple observation that 
if $\tau$ is an analytic expanding circle map which is holomorphically
expansive on an annulus $A_{r,R}$, then the corresponding transfer
operator $\lt_\tau$ is an endomorphism of $H^2(A_{r,R})$. In fact, as
we shall see later, $\lt_\tau$ has much stronger functional analytic
properties on $H^2(A_{r,R})$. 

The next fact is concerned with the strong dual $H^2(A_{r,R})^\prime$ of $H^2(A_{r,R})$, 
that is, the space of continuous linear functionals on $H^2(A_{r,R})$ 
equipped with the topology of uniform convergence on the unit ball. It
turns out that it can be represented 
in terms of the topological direct sum
$H^2(D_r)\oplus \HRnull$, equipped with the norm 
$\norm{(h_1,h_2)}{}^2 = \norm{h_1}{H^2(D_r)}^2+\norm{h_2}{\HRnull}^2$, turning it into a 
Hilbert space. 

\begin{propo}\label{lem:iso_new}
The dual space $H^2(A_{r,R})^\prime$ 
is isomorphic to $H^2(D_r)\oplus \HRnull$ with the isomorphism
given by
\begin{align*}  
      H^2(D_r)\oplus \HRnull & \rightarrow  H^2(A_{r,R})^\prime\\ 
      (h_1, h_2) & \mapsto l\,,
\end{align*}
where
\begin{equation}\label{eq:lKz}
 l(f) = \frac{1}{2\pi i} \int_{\partial_+D_r} f(z)h_1(z) \,dz + 
\frac{1}{2\pi i} \int_{\partial_+D_R} f(z)h_2(z) \,dz \quad (f\in H^2(A_{r,R}))\,.
\end{equation}
\end{propo}

\begin{proof}
See \cite{BJS2} for a short proof.  See also 
\cite[Proposition 3]{Royden1988}, where similar representations for the duals of Hardy spaces 
over multiply connected regions are provided.   
\end{proof}

Next we note that, given a circle map $\tau$, we
can associate with it the corresponding \emph{composition operator} $C_\tau$
defined for $f:\T\to \C$ by 
\[ C_\tau f = f\circ \tau\,, \]
which, in the context of dynamical systems, is also known as
\emph{Koopman operator}. 
Moreover, if $\tau$ is an
analytic expanding circle map which is holomorphically expansive on an
annulus $A_{r,R}$, 
then $\lt_\tau^\prime$, 
the Banach space-adjoint\footnote{Recall that this means that  
$\lt_\tau^\prime: H^2(A_{r,R})^\prime\to  H^2(A_{r,R})^\prime$ is given by 
$(\lt_\tau^\prime l)(f) = l(\lt_\tau f)$ for all $l\in H^2(A_{r,R})^\prime$ and  $f\in H^2(A_{r,R})$.}
of the corresponding transfer operator $\lt_\tau$, can be represented 
as a compression of $C_\tau$ to $H^2(D_r)\oplus \HRnull$. 

In order to make this connection more precise, we need to introduce
certain projection operators on  $L^2(\mathbb{T_\rho})$. 
For any $f\in L^2(\mathbb{T_\rho})$ we can write 
$f(z) = \sum_{n\in \mathbb{Z}} f_n z^n$,
so that $f = f_+ + f_-$ with 
$f_+(z)=\sum_{n=0}^{\infty} f_n z^n$ and  
$f_-(z)=\sum_{n=1}^{\infty} f_{-n} z^{-n}$. 
Since $\norm{f}{L^2(\mathbb{T_\rho})}^2 = 
\sum_{n=-\infty}^{\infty} |f_n|^2 \rho^{2n}<\infty,$ the functions $f_{+}$ and $f_{-}$ can
be viewed as functions in $H^2(D_\rho)$ and 
$H^2_0(D^\infty_\rho)$, respectively. Thus, we can 
define two projection operators 
$\Pi^{(\rho)}_+\colon L^2(\mathbb{T}_\rho)\to H^2(D_\rho)$ and
$\Pi^{(\rho)}_-\colon L^2(\mathbb{T}_\rho)\to H^2_0(D^\infty_\rho)$ by setting
\begin{equation}\label{eq:P+P-}
\Pi^{(\rho)}_+f = f_{+}\quad \text{and} \quad  \Pi^{(\rho)}_-f = f_{-}\,,
\end{equation}
which are easily seen to be bounded. 
The representation alluded to above can now be stated as follows.  
 
\begin{propo}\label{prop:adj}
Let $\tau$ be an analytic expanding circle map which is holomorphically
expansive on $A_{r,R}$ and let $\lt_\tau \colon H^2(A_{r,R})\to
H^2(A_{r,R})$ 
be the corresponding transfer
operator. Then, using the isomorphism given in Proposition \ref{lem:iso_new}, the
adjoint $\lt_\tau^\prime$ can be represented by the  
operator 
\[\lt_\tau^\dagger\colon \monster \to \monster\,,\]
where 
\begin{equation}
\label{eq:lprep:orpr}
   \lt_\tau^\dagger=
    \begin{pmatrix}
        \Pi^{(r)}_+C_\tau    & \Pi^{(R)}_{+} C_\tau \\
        \Pi^{(r)}_{-} C_\tau  & \Pi^{(R)}_{-} C_\tau  \\
    \end{pmatrix}
\end{equation}
if $\tau$ is orientation preserving and 
\begin{equation}
\label{eq:lprep:orre}
  \lt_\tau^\dagger:=-
    \begin{pmatrix}
        \Pi^{(R)}_+C_\tau    & \Pi^{(r)}_{+} C_\tau \\
        \Pi^{(R)}_{-} C_\tau & \Pi^{(r)}_{-} C_\tau  \\
    \end{pmatrix}
\end{equation}
 if $\tau$ is orientation reversing.  
\end{propo}

\begin{proof}

See \cite{BJS2}.   
\end{proof}
%\rev{\footnote{\rev{Note that
%$H^2(D_R)$ can be viewed as a subspace of $H^2(D_r)$, and similarly
%$H^2_0(D^\infty_r)$ as a subspace of $H^2_0(D^\infty_R)$.
%Thus the off-diagonal elements of $\tr'$ are well defined.}}} 

We now make an important observation: the operator $\lt_\tau^\dagger$ 
makes sense even if $\tau$ does not preserve the
unit circle, but is merely holomorphically expansive. In fact, as we
shall see shortly, the operator $\lt_\tau^\dagger$ has strong spectral
properties for any holomorphically expansive $\tau$.  

These spectral properties are conveniently described in terms of
the theory of exponential classes developed in \cite{expoclass}, which
we now briefly outline. Recall that if $L:H\to H$ is a compact operator on a
Hilbert space $H$, we use $(\lambda_n(L))_{n\in \N}$ to denote its
eigenvalue sequence, counting algebraic multiplicities and ordered by
decreasing modulus so that 
\[ \abs{\lambda_1(L)}\geq \abs{\lambda_2(L)} \geq \cdots \]
If $L$ has only finitely man non-zero eigenvalues, we set
$\lambda_n(L)=0$ for $n>N$, where $N$ denotes the number of non-zero
eigenvalues of $L$. Furthermore, for $L:H_1\to H_2$ a compact operator
between Hilbert spaces $H_1$ and $H_2$ and $n\in \N$, the  
\emph{$n$-th singular value} of $L$ is given by 
\[ s_n(L)=\sqrt{\lambda_n(L^*L)} \quad (n \in \N)\,, \]
where $L^*$ denotes the Hilbert space adjoint of $L$. 
 
\begin{defi}
If a compact operator $L$ between Hilbert spaces
satisfies  
\[ s_n(L)\leq c_1\exp(-c_2n) \quad (\forall n\in \N) \] 
for some constants $c_1,c_2>0$ we say that \emph{$L$ is of exponential class}.  
The collection of all compact operators of exponential class will be
denoted by $\mathcal E$. 
\end{defi}

Standard examples of operators of exponential class are embeddings
between Hardy spaces, as the following lemma shows. 

\begin{lem}
\label{lem:Jexpoclass}
  Let $0<r'<r$ and let $J:H^2(D_r)\to H^2(D_{r'})$ denote the canonical
  embedding given by $(Jf)(z)=f(z)$ for $f\in H^2(D_r)$ and $z\in
  D_{r'}$. Then $J$ is of exponential class.  
\end{lem}

\begin{proof} In order to see this note that 
\[
   (J^*Je_n^{(r)},e_m^{(r)})_{H^2(D_r)}
    =(Je_n^{(r)},Je_m^{(r)})_{H^2(D_{r'})}
    =\delta_{nm}\left ( \frac{r'}{r} \right )^{2n} 
\]
where $(e_n^{(r)})_{n\in N_0}$ denotes the orthonormal basis of
$H^2(D_r)$ given in Remark~\ref{rem:Basis}. 
Thus 
\[ s_n(J)=\left ( \frac{r'}{r} \right )^{n-1}\,, \]
and it follows that $J\in \mathcal E$.  
\end{proof}

\begin{rem}
\label{rem:Jexpoclass}
  A similar argument shows that for $0<R<R'$ the canonical embedding 
   $J:H^2_0(D_{R})\to H^2_0(D_{R'})$ is of exponential class.
\end{rem}

Other examples of naturally occurring operators of exponential class can be
found in \cite{bandtlow_chu, bj_advances, bj_etds}. In fact, as we
shall see shortly, the operator $\lt^\dagger_\tau$ is of exponential
class for every holomorphically expansive $\tau$, and, moreover, its
eigenvalue sequence decays at an exponential rate. The proof of these
results relies on the following properties of $\mathcal E$. 

\begin{propo}
\label{prop:expoclassprop}
The exponential class $\mathcal E $ is a two-sided operator ideal, that is, 
the following two properties hold: 
\begin{enumerate}
\item $L_1,L_2\in \mathcal{E}$ and $\alpha_1,\alpha_2\in \C$
  imply $\alpha_1L_1+\alpha_2L_2\in \mathcal{E}$, whenever this 
  linear combination is defined;
\item if $L_1,L_3$ are bounded operators and $L_2\in \mathcal{E}$
  then $L_1L_2L_3\in \mathcal{E}$, whenever this 
  product is defined. 
\end{enumerate}
Moreover, if $L\in \mathcal E$ is an endomorphism then 
\[ \abs{\lambda_n(L)}\leq c_1\exp(-c_2n) \quad (\forall n \in \N) \]
for some constants $c_1,c_2>0$. 
\end{propo}

\begin{proof}
See \cite[Propositions 2.8 and 2.10]{expoclass}  
\end{proof}

% Before proving that $\lt^\dagger_\tau \in \mathcal{E}$ for any 
% holomorphically expansive $\tau$ it will be useful to group them into two classes. 
% 
% 
% \begin{defi}
 %  Let $\tau$ be holomorphically expansive on $A_{r,R}$. If 
% \[ \tau(\T_r) \subset D_r \text{ and }   
 %   \tau(\T_R)\subset D_R^\infty \]  
% we shall say that $\tau$ is \emph{orientation preserving}, while if 
% \[ \tau(\T_r) \subset D_R^\infty \text{ and }   
 %   \tau(\T_R)\subset D_r \]  
% we shall say that $\tau$ is \emph{orientation reversing}. 
% \end{defi}

We shall now show that the entries of the matrix defining $\lt^\dagger_\tau$ 
in (\ref{eq:lprep:orpr}) and (\ref{eq:lprep:orre}) 
are well-defined and are each of exponential class for 
any holomorphically expansive  
$\tau$. 

\begin{propo}
\label{prop:Cexpoclass}
  Let $\tau$ be holomorphically expansive on $A_{r,R}$.
  \begin{enumerate}
  \item \label{enum:Cexpoclass:a}
  If $\tau$ is orientation preserving then 
  \[ C_\tau(H^2(D_r))\subset L^2(\T_r) \text{ and } 
     C_\tau(H^2(D_R^\infty))\subset L^2(\T_R)\,. \] 
  \item \label{enum:Cexpoclass:b}
  If $\tau$ is orientation reversing then 
  \[ C_\tau(H^2(D_r))\subset L^2(\T_R) \text{ and } 
     C_\tau(H^2(D_R^\infty))\subset L^2(\T_r)\,. \] 
  \end{enumerate}
  Moreover in both cases, the restrictions 
  $C_\tau|H^2(D_r)$ and $C_\tau|H^2(D_R^\infty)$ are of exponential class. 
\end{propo}

\begin{proof}
We shall only prove case (\ref{enum:Cexpoclass:b}); 
the other one is similar. We start
by observing that, since
$\tau$ is orientation reversing, we can choose $0<r'<r$ such that 
\[ \tau(\T_R)\subset D_{r'}\subset D_r\,. \]
We shall now show that $C_\tau$ maps $H^2(D_{r^\prime})$ continuously to
$L^2(\T_R)$. In order to see this note that $\tau(\T_R)$ is a compact
subset of $D_{r'}$, so 
\[ C:=\sup_{z\in \T_R} \frac{r'}{\sqrt{(r')^2-\abs{\tau(z)}^2}}<\infty\,. \] 
Thus, using Lemma~\ref{lem:pevalcont}, we have for any $f\in H^2(D_{r^\prime})$ 
\[ \sup_{z\in \T_R} \abs{f(\tau(z))}\leq
    C \norm{f}{H^2(D_{r^\prime})}\,, \]
and so 
\begin{equation}
\label{eq:Ccont}
\norm{C_\tau f}{L^2(\T_R)} \leq 
    C \norm{f}{H^2(D_{r^\prime})} \quad (\forall f \in
    H^2(D_{r^\prime})) \,.
\end{equation}
We now observe that $C_\tau|H^2(D_r)$ admits a factorisation of the
form 
\[ C_\tau|H^2(D_r)=C_\tau|H^2(D_{r^\prime})J\,, \]
where $J$ denotes the canonical embedding of $H^2(D_r)$ in
$H^2(D_{r^\prime})$. Thus, since $C_\tau:H^2(D_{r^\prime})\to
L^2(\T_R)$ is continuous by (\ref{eq:Ccont}) it follows that 
$C_\tau$ maps $H^2(D_r)$ continuously to
$L^2(\T_R)$. Moreover, since $J$ is of exponential class by
Lemma~\ref{lem:Jexpoclass}, Proposition~\ref{prop:expoclassprop} now implies that 
$C_\tau|H^2(D_r)$ is of exponential class as well. 

A similar argument using Remark~\ref{rem:Jexpoclass} instead of 
Lemma~\ref{lem:Jexpoclass} shows that
$C_\tau$ maps $H^2(D_R^\infty)$ continuously to $L^2(\T_r)$ and that 
$C_\tau|H^2(D_R^\infty)$ is of exponential class.  
\end{proof}

\begin{cor}
\label{co:Cexpoclass}
If $\tau$ is holomorphically expansive then 
$\lt_\tau^\dagger$ 
is of exponential class and, in particular, its eigenvalue sequence decays
exponentially.   
\end{cor}

\begin{proof}
  Follows from Propositions~\ref{prop:expoclassprop} and \ref{prop:Cexpoclass}
\end{proof}
\end{section}

\begin{rem}
In the corollary above and in the following, we shall always tacitly assume that if 
$\tau$ is holomorphically expansive on $A_{r,R}$, then $\lt_\tau^\dagger$ will be 
considered as an operator from $\monster$ to $\monster$. 
\end{rem}

We finish this section with a result that will allow us to calculate the eigenvalue 
sequence of a particular class of analytic expansive circle maps.

\begin{propo}
\label{prop:ldtrace}
  Let $\tau$ be holomorphically expansive on $A_{r,R}$. Then
  $\lt_\tau^\dagger$ is trace class and its trace is given by 
\[ \Tr(\lt_\tau^\dagger) = \frac{\omega(\tau)}{2\pi i} \int_{{\partial_+} A_{r,R}}
\frac{1}{\tau(z)-z}\,dz\,. \] 
% where ${\partial_+} A_{r,R}$ denotes the positively oriented boundary of
% $A_{r,R}$. 
\end{propo}

\begin{proof}
Clearly, $\lt_\tau^\dagger$ is trace class, since by 
Corollary~\ref{co:Cexpoclass} it is of exponential class. In order to
calculate its trace we observe that if $\rho\in [r,R]$ and $n\in \Z$ we have 
\begin{equation}
\label{eq:cee}
  (C_\tau e_n^{(\rho)}, e_n^{(\rho)})_{L^2(\T_\rho)}
     = \frac{1}{2\pi} \int_0^{2\pi} 
          \frac{\tau(\rho e^{i\theta})^n}{\rho^n}
                e^{-in\theta}\, d\theta 
     =\frac{1}{2\pi i} \int_{{\partial_+} D_\rho} 
          \frac{\tau(z)^n}{z^{n+1}}\, dz\,. 
\end{equation}
Suppose now that $\tau$ is orientation preserving. Then, using the
cyclicity of the trace, we have  
\[
  \Tr(\lt_\tau^\dagger)=
   \Tr(\Pi_+^{(r)}C_\tau)+\Tr(\Pi_-^{(R)}C_\tau)=
   \Tr(C_\tau \Pi_+^{(r)})+\Tr(C_\tau \Pi_-^{(R)})\,.
\]
But by (\ref{eq:cee})  
\begin{multline*}
   \Tr(C_\tau \Pi_+^{(r)})
    =\sum_{n=-\infty}^\infty (C_\tau \Pi_+^{(r)}e_n^{(r)},
          e_n^{(r)})_{L^2(\T_r)}= \\
    =\sum_{n=0}^\infty \frac{1}{2\pi i} \int_{{\partial_+} D_r} 
          \frac{\tau(z)^n}{z^{n+1}}\, dz
    = \frac{1}{2\pi i} \int_{{\partial_+} D_r} 
          \frac{1}{z-\tau(z)}\, dz\,,
\end{multline*}
and 
\begin{multline*}
   \Tr(C_\tau \Pi_-^{(R)})
     =\sum_{n=-\infty}^\infty (C_\tau \Pi_-^{(R)}e_n^{(R)},
          e_n^{(R)})_{L^2(\T_R)}= \\
    =\sum_{n=1}^\infty \frac{1}{2\pi i} \int_{{\partial_+} D_R} 
          \frac{z^{n-1}}{\tau(z)^{n}}\, dz 
    =\frac{1}{2\pi i} \int_{{\partial_+} D_R^\infty} 
          \frac{1}{z-\tau(z)}\, dz\,,
\end{multline*}
and the assertion follows by observing that 
$\partial_+A_{r,R}=\partial_+D_r^\infty \cup \partial_+D_R$. 
The proof for orientation reversing $\tau$
is similar. 
\end{proof}

\section{Blaschke and anti-Blaschke products}
\label{sec:Blaschke}
In this section we shall consider a particular class of analytic circle maps, for 
which the eigenvalue sequence of the associated transfer operators can be calculated 
exactly. 

\begin{defi}
For $d\in \N$ let $a=(\alpha, a_1,\ldots,a_d)$ be a $(d+1)$-tuple of 
complex numbers with $\alpha \in \T$ and $a_1,\ldots, a_d \in
\mathbb{D}$. Then  
\[
  B_a(z)= \alpha\prod_{j=1}^{d} \frac{z-a_j}{1-\overline{a_j}z}
\]
is called a \emph{Blaschke product of degree $d$} or a 
\emph{finite
Blaschke product}. 
\end{defi}

We shall now collect a number of facts about Blaschke products. 

\begin{propo} 
\label{prop:blbasic} 
  Let $B_a$ be a finite Blaschke product. Then the
  following holds. 
\begin{enumerate}
 \item \label{eq:bl1}
  $B_a$ is meromorphic on $\hat{\mathbb{C}}$ and holomorphic 
     on $\cl{\mathbb D}$. 
 \item \label{eq:bl2}
 $B_a$ leaves both $\mathbb{T}$ and $\mathbb{D}$ invariant.  
 \item \label{eq:bl3}
       We have $B_a(z^{-1})=B_{\overline{a}}(z)^{-1}$, where 
       $\overline{a}=(\overline{\alpha}, \overline{a_1},\ldots, \overline{a_d})$.  
 \item \label{eq:bl3.5}
   $B_a$ is analytic expanding if and only if it is holomorphically
   expansive. 
 \item \label{eq:bl4}
  If $\sum_{j=1}^{d}(1-|a_j|)/(1+|a_j|)>1$ then $B_a$ is 
  holomorphically expansive. 
 \item \label{eq:bl5}
   If $B_a$ is holomorphically expansive, then $B_a$ has a unique
   fixed point $z_0$ in $\mathbb{D}$ and the corresponding multiplier 
  $B_a'(z_0)$ belongs to $\mathbb D$.  
\end{enumerate}
\end{propo}
\begin{proof} Part (\ref{eq:bl1}) is clear, while (\ref{eq:bl2}) and (\ref{eq:bl3})
follow from a short calculation. For (\ref{eq:bl3.5}) see
\cite[Theorem~1]{tischler1999} and for 
(\ref{eq:bl4}) see \cite[Corollary to Proposition~1]{martin1983}. Finally, (\ref{eq:bl5}) 
follows by 
combining  \cite[Proposition~2.1]{pujals2006} and \cite[Theorem~1]{tischler1999}.
\end{proof}

Part (\ref{eq:bl1}) and  (\ref{eq:bl2}) of the above proposition show that a 
finite Blaschke product yields an analytic circle map. 
Curiously enough, 
any analytic circle map which is also holomorphic on $\cl{\mathbb{D}}$ is necessarily a finite 
Blaschke product (see, for example, \cite[Exercise
6.12]{Burckel}). In
particular, 
the composition of two finite Blaschke products is again a 
finite Blaschke product. 

It turns out that for expanding circle maps arising from Blaschke
products a complete determination of the spectra of the associated
transfer operators is possible. 
For certain Blaschke products of degree 2
this is shown in \cite{SBJ_nonlinearity} relying on a block-diagonal matrix
representation of the transfer operator. The general case is discussed
in \cite{BJS2} using the spectral theory of composition operators with
holomorphic symbols. Below we rederive this result by yet another
method, exploiting the fact that the trace of $\lt_{B_a}^\dagger$ is 
easily calculated whenever $B_a$ is a holomorphically expansive
Blaschke product.  
% In the past literature, we point out the work of Levin et al where an explicit formula for Fredholm determinants
% of certain transfer operators related to rational maps is proved 
% \cite{LevinSodin}.

\begin{lem}
\label{lem:blaschketrace}
Let $B_a$ be a finite Blaschke product which is holomorphically
expansive on $A_{r,R}$. Then 
\[ \Tr(\lt_{B_a}^\dagger)=1+\frac{\mu}{1-\mu} +
\frac{\overline{\mu}}{1-\overline{\mu}} \,, \]
where $\mu$ is the multiplier of the fixed point of $B_a$ in $\mathbb
D$. 
\end{lem}

\begin{proof} Let $z_0$ denote the fixed point of $B_a$ in $\mathbb D$. 
 Since $B_a$ is holomorphically expansive on $A_{r,R}$ we must have $z_0\in D_r$. Thus 
\[
\frac{1}{2\pi i} \int_{{\partial_+} D_r} \frac{1}{z-B_a(z)}\,dz=\frac{1}{1-B'_a(z_0)}\,.
\]
Furthermore, changing variables and using part 
(\ref{eq:bl3}) of Proposition~\ref{prop:blbasic} we have 
\begin{multline}
 \frac{1}{2\pi i} \int_{{\partial_+} D_R^\infty} \frac{1}{z-B_a(z)}\,dz
  =-\frac{1}{2\pi i} \int_{{\partial_+} D_{R^{-1}}}
        \frac{1}{z^{-1}-B_a(z^{-1})}\frac{1}{z^2}\,dz=\\
   =-\frac{1}{2\pi i} \int_{{\partial_+} D_{R^{-1}}}
     \frac{1}{z^{-1}-B_{\overline{a}}(z)^{-1}} \frac{1}{z^2}\,dz
   =\frac{1}{2\pi i} \int_{{\partial_+} D_{R^{-1}}} 
  \left ( \frac{1}{z-B_{\overline{a}}(z)}-\frac{1}{z} \right )\,dz\,.
\end{multline}
It is not difficult to see that the unique fixed point of 
$B_{\overline{a}}$ in the unit disk is $\overline{z_0}$ and that   
$\overline{z_0}\in D_{R^{-1}}$. Moreover,
it follows that $B'_{\overline{a}}(\overline{z_0})=\overline{B'_a(z_0)}$. 
Thus
\[ \frac{1}{2\pi i} \int_{{\partial_+} D_{R^{-1}}} 
  \left ( \frac{1}{z-B_{\overline{a}}(z)}-\frac{1}{z} \right )\,dz
   =\frac{1}{1-\overline{B'_a(z_0)}}-1\,,
\]
and the desired formula follows from Proposition~\ref{prop:ldtrace}. 
\end{proof}

Since the trace on a Hilbert space is spectral,
that is, it coincides with the sum of eigenvalues (see, for example, 
\cite[4.7.15]{pietsch}), the eigenvalues of a trace class operator $L$
are given by the reciprocals of the zeros of the corresponding
spectral determinant $z\mapsto \det(I-zL)$, an entire function given
by 
\[ \det(I-zL)=\exp \left ( - \sum_{n=1}^\infty \frac{z^n}{n}
  \Tr(L^n)\right )\]
for $z$ in a small neighbourhood of $0$ (see, for example, 
\cite[4.6.2]{pietsch}). We are now able to 
calculate the eigenvalue sequence of the transfer operator associated with a
holomorphically expansive Blaschke product. 

\begin{propo}
\label{product1}
  Let $B_a$ be a finite Blaschke product which is holomorphically
expansive on $A_{r,R}$. Then 
\[ \det(1-z\lt_{B_a}^\dagger)=(1-z)\prod_{k=1}^\infty (1-\mu^kz)(1-\overline{\mu}^kz)\,,\]
where, as before, $\mu$ is the multiplier of the fixed point of $B_a$ in the unit disk. 
In particular, the eigenvalue sequence of $\lt_{B_a}^\dagger$ is given by 
\[ 
\lambda_n(\lt_{B_a}^\dagger)=
 \begin{cases}
 \mu^{n/2} & \text{for $n$ even;} \\
 \overline{\mu}^{(n-1)/2} & \text{for $n$ odd.}
 \end{cases}
\] 
\end{propo}
\begin{proof}
First we observe that the multiplier of the fixed 
point in the unit disk 
of $B_a^n$, the $n$-th iterate of $B_a$, 
is
$\mu^n$. Lemma~\ref{lem:blaschketrace} now implies 
\[ \Tr((\lt_{B_a}^\dagger)^n)
    =\Tr(\lt_{B_a^n}^\dagger)
    =1+\frac{\mu^n}{1-\mu^n} +  
      \frac{\overline{\mu}^n}{1-\overline{\mu}^n}\,. \]  
Thus, for $z\in \mathbb{D}$ we have 
\begin{align*}
\log \det( 1-z\lt_{B_a}^\dagger) 
    & = -\sum_{n=1}^\infty \frac{z^n}{n}  \Tr((\lt_{B_a}^\dagger)^n) \\ 
    & = -\sum_{n=1}^\infty \frac{z^n}{n}\left ( 1+\frac{\mu^n}{1-\mu^n} +  
      \frac{\overline{\mu}^n}{1-\overline{\mu}^n} \right ) \\ 
   & = -\sum_{n=1}^\infty \frac{z^n}{n}
       -\sum_{k=1}^\infty \sum_{n=1}^\infty \frac{1}{n} \mu^{kn}z^n 
       -\sum_{k=1}^\infty \sum_{n=1}^\infty \frac{1}{n} \overline{\mu}^{kn}z^n \\
   & = \log(1-z) +  \sum_{k=1}^\infty \log(1-\mu^kz) + 
       \sum_{k=1}^\infty \log(1-\overline{\mu}^kz) \,,  
\end{align*}
and the assertions follow. 
\end{proof}

\begin{rem}
\label{rem:blexpodecay}
The proposition above makes it possible to manufacture analytic 
expanding circle maps of a given degree $d\geq 2$ so that the decay of the 
eigenvalue sequence of the corresponding transfer operator is 
exactly exponential. To be precise, let $d\geq 2$ and let 
$a=(1,0,a_2,\ldots,a_d)$ with $a_2,\ldots,a_d\neq 0$. Using (\ref{eq:bl4}) 
of Proposition~\ref{prop:blbasic} it follows that $B_a$ yields 
an analytic expanding circle map, which is easily seen to be of degree $d$. 
Moreover, the unique fixed point of $B_a$ in the unit disk is 
$0$ and the corresponding  multiplier $\mu=\prod_{j=2}^d(-a_j)$ is non-zero. Thus   
the above proposition implies that 
\[ \lim_{n\to \infty}\abs{\lambda_{n}(\lt_{B_a})}^{1/n}=\sqrt{\abs{\mu}}\,. \]
\end{rem}

We now turn our attention to anti-Blaschke products, 
which are defined as follows. 

\begin{defi}
If $B_a$ is Blaschke product of degree $d$ then 
\[ \hat{B}_a(z)=\frac{1}{B_a(z)}\,, \]
will be called an  
\emph{anti-Blaschke product of degree $d$} or 
a \emph{finite anti-Blaschke product}. 
\end{defi}

\begin{rem}
  Note that Blaschke products yield orientation preserving circle maps, 
while anti-Blaschke products provide examples of orientation reversing circle maps.  
\end{rem}
For later use, we note the following properties of the second iterate
of a finite anti-Blaschke product.  
\begin{lem}
\label{Julialem}
Let $B_a$ be a Blaschke product and let 
$\hat{B}_a$ denote the corresponding anti-Blaschke product. 
If $\hat{B}_a$ is holomorphically expansive, 
then $\hat{B}_a\circ \hat{B}_a$ is a holomorphically expansive
(ordinary) Blaschke product with a 
unique fixed point $z_0\in \mathbb{D}$, 
% which is also characterised
% as the unique point in $\mathbb{D}$ satisfying $B_a(z_0)=\overline{z_0}$. 
the multiplier of which satisfies 
\[ (\hat{B}_a\circ
\hat{B}_a)^\prime(z_0)=\abs{B_a^\prime(z_0)}^2\,.
\]
\end{lem}

\begin{proof}
Let $B_a$ be a holomorphically expansive Blaschke product and let 
$\hat{B}_a$ denote the corresponding holomorphically expansive 
anti-Blaschke product. 
We start by observing that by 
(\ref{eq:bl3}) of Proposition~\ref{prop:blbasic}
we have 
\begin{equation}
\label{antiB1}
\hat{B}_a(\hat{B}_a(z))=B_a(B_a(z)^{-1})^{-1}=B_{\overline{a}}(B_a(z))\,. 
\end{equation}
Thus, by (\ref{eq:bl3.5}) of Proposition~\ref{prop:blbasic}, the second
iterate $\hat{B}_a\circ \hat{B}_a$ is a holomorphically expansive Blaschke
product, which, by (\ref{eq:bl5}) of Proposition~\ref{prop:blbasic}, 
has a unique fixed point $z_0\in \mathbb{D}$. 

We shall now show that $z_0$ is the unique point
in $\mathbb{D}$ satisfying 
\begin{equation}
\label{antiB2}
B_a(z_0)=\overline{z_0}\,.
\end{equation}
In order to
see this, note that 
\begin{align*}
  (B_{\overline{a}}\circ B_a)(z_0) & =z_0  \\
   (B_a\circ B_{\overline{a}})(B_a(z_0)) & =B_a(z_0)\,,\\
\end{align*}
which implies that $B_a(z_0)\in \mathbb{D}$ is the unique fixed point
in $\mathbb{D}$ of the holomorphically expansive Blaschke product
$B_a\circ B_{\overline{a}}$. 
At the same time we have 
\[ (B_a\circ B_{\overline{a}})(\overline{z_0})
  =B_a(\overline{B_a(z_0)})
  =\overline{ B_{\overline{a}}(B_a(z_0))} = \overline{z_0}\,,
\]
so $B_a(z_0)=\overline{z_0}$, as claimed.  

Now, using (\ref{antiB1}) and (\ref{antiB2}) 
we see that 
\begin{multline*}
    (\hat{B}_a\circ \hat{B}_a)^\prime(z_0)
    =(B_{\overline{a}}\circ B_a)^\prime (z_0))
    =B_{\overline{a}}^\prime (B_a(z_0))B_a^\prime (z_0)\\
    =B_{\overline{a}}^\prime (\overline{z_0})B_a^\prime (z_0)
    =\overline{B_{a}^\prime(z_0)}B_{a}^\prime(z_0)
    =|B_{a}^\prime(z_0)|^2\,,
\end{multline*}
and the remaining claim of the lemma follows. 
\end{proof}

As for Blaschke products, the analytic structure of 
anti-Blaschke products makes it possible to calculate the traces of the 
corresponding transfer operators. 

\begin{lem}
\label{lem:antiblaschketrace}
Let $\hat{B}_a$ be a finite anti-Blaschke product which is 
a holomorphically expansive on $A_{r,R}$. 
Then 
\[ \Tr(\lt_{\hat{B}_a}^\dagger)=1 \,. \]
\end{lem}

\begin{proof} Let $B_a=\hat{B}_a^{-1}$ denote the corresponding 
Blaschke product. Since 
$zB_a(z)\in \mathbb{D}$ whenever $z\in \mathbb{D}$, we have
\[
  \frac{1}{2\pi i} 
 \int_{{\partial_+} D_r} \frac{1}{z-B_a(z)^{-1}}\,dz
=\int_{{\partial_+} D_r} \frac{B_a(z)}{zB_a(z)-1}\,dz=0\,. 
\]
Furthermore, changing variables and using
(\ref{eq:bl3}) of Proposition~\ref{prop:blbasic}
we have 
\begin{multline*}
 \frac{1}{2\pi i} \int_{{\partial_+} D_R^\infty} \frac{1}{z-B_a(z)^{-1}}\,dz
  =-\frac{1}{2\pi i} \int_{{\partial_+} D_{R^{-1}}}
        \frac{1}{z^{-1}-B_a(z^{-1})^{-1}}\frac{1}{z^2}\,dz=\\
   =-\frac{1}{2\pi i} \int_{{\partial_+} D_{R^{-1}}}
     \frac{1}{z^{-1}-B_{\overline{a}}(z)} \frac{1}{z^2}\,dz
   =\frac{1}{2\pi i} \int_{{\partial_+} D_{R^{-1}}} 
  \left ( \frac{B_{\overline{a}}(z)}{zB_{\overline{a}}(z)-1}-\frac{1}{z} \right )\,dz=-1\,.
\end{multline*}
Since $\tau$ is orientation reversing, the assertion follows from 
Proposition~\ref{prop:ldtrace}. 
\end{proof}

As before, we are now able to determine the eigenvalue sequence of 
transfer operators corresponding to anti-Blaschke products. 

\begin{propo}
\label{product2}
Let $\hat{B}_a$ be a finite anti-Blaschke product 
which is holomorphically expansive on $A_{r,R}$. Then 
\[ \det(1-z\lt_{\hat{B}_a}^\dagger)=
(1-z)\prod_{k=1}^\infty (1-\mu^kz)(1+\mu^kz)\,,\]
where $\mu\in [0,1)$ is the square root of the multiplier of
the fixed point of $\hat{B}_a\circ \hat{B}_a$ in the unit disk 
(guaranteed by Lemma~\ref{Julialem}). 

In particular, the eigenvalue sequence of $\lt_{\hat{B}_a}^\dagger$ is given by 
\[ 
\lambda_n(\lt_{\hat{B}_a}^\dagger)=
 \begin{cases}
 -\mu^{n/2} & \text{for $n$ even;} \\
 \hphantom{-}\mu^{(n-1)/2} & \text{for $n$ odd.}
 \end{cases}
\] 
\end{propo}
\begin{proof}
We start by observing that by Lemma~\ref{Julialem} the 
even iterates of $\hat{B}_a$ 
are iterates of a finite Blaschke product, 
while odd iterates are anti-Blaschke products. 
Moreover, the multiplier of the fixed point in the unit disk of the
$2n$-th iterate of $\hat{B}_a$ is $\mu^{2n}$. 

Thus, using Lemmas \ref{lem:blaschketrace} and  \ref{lem:antiblaschketrace} 
it follows that   
\[ \Tr((\lt_{\hat{B}_a}^\dagger)^n)
   =\begin{cases}
      1     & \text{ for $n$ odd;} \\
      1+ \frac{2\mu^n}{1-\mu^n} &  
            \text{ for $n$ even.} 
     \end{cases}
\]  
Thus, for $z\in \mathbb{D}$ we have 
\begin{align*}
\log \det( 1-z\lt_{\hat{B}_a}^\dagger) 
    & = -\sum_{n=1}^\infty \frac{z^n}{n}  \Tr((\lt_{\hat{B}_a}^\dagger)^n) \\ 
    & = -\sum_{n=1}^\infty \frac{z^n}{n} - 
         \sum_{n=1}^\infty \frac{z^{2n}}{2n}  
      \frac{2\mu^{2n}}{1-\mu^{2n}} \\ 
    & = -\sum_{n=1}^\infty \frac{z^n}{n} - 
         \sum_{k=1}^\infty \sum_{n=1}^\infty \frac{1}{n}  
         \mu^{2nk}z^{2n} \\ 
    & = \log(1-z) +  \sum_{k=1}^\infty \log(1-\mu^{2k}z^2)\,, 
\end{align*}
and the assertions follow. 
\end{proof}

\begin{rem}
\label{rem:ablexpodecay}
Arguing as in 
Remark~\ref{rem:blexpodecay}, the proposition above allows us to construct  
orientation reversing analytic expanding circle maps of a given degree 
$d\leq -2$ so that the decay of the 
eigenvalue sequence of the corresponding transfer operator is 
exactly exponential. 
Let $d\geq 2$ and let 
$a=(1,0,a_2,\ldots,a_d)$ with $a_2,\ldots,a_d\neq 0$. As in Remark~\ref{rem:blexpodecay}, the 
corresponding Blaschke product $B_a$ is holomorphically expansive, and so is the associated anti-Blaschke 
product $\hat{B}_a$. Moreover, the unique fixed point of $B_a$ in $\mathbb{D}$ is 
$0$ and the corresponding multiplier is $\prod_{j=2}^d(-a_j)$. 
Since $B_a(0)=B_{\overline{a}}(0)=0$, equation (\ref{antiB1}) 
implies that $0$ is the 
unique fixed point of $\hat{B}_a\circ \hat{B}_a$ in $\mathbb{D}$ 
and that the corresponding multiplier is given by 
$\prod_{j=2}^d|a_j|^2$.
It now follows that the anti-Blaschke product $\hat{B}_a$ is 
an orientation reversing 
analytic expanding circle map of degree $-d$ such that the eigenvalues of the corresponding transfer operator 
satisfy 
\[ \lim_{n\to \infty}\abs{\lambda_{n}
  (\lt_{\hat{B}_a})}^{1/n}=\sqrt{\mu}\,, \]
where $\mu=\prod_{j=2}^d|a_j|$. 
\end{rem}

\begin{rem}
It is rather curious that while the eigenvalues of the transfer operators
associated with Blaschke products can have non-vanishing imaginary
parts, this is not the case for the eigenvalues of the transfer
operators associated with anti-Blaschke products, which are, as the
above proposition shows, always real.  
\end{rem}

\begin{section}{Potential theoretic tools}
\label{sec:potentialth}

In this section we collect some basic definitions and recall, 
mostly without proofs, the material necessary to prove the main result.
Our references are \cite{LelongGruman} for the theory of 
several complex variables and \cite{Ransford} for 
potential theory in the one dimensional case. In our applications, we
mainly need to look at the case of one and two 
complex variables, but we state the results in the $n$-dimensional case.
Let $\Oo \subset \C^n$ be an open connected non-empty set. 
We denote by $\Delta(a,r)\subset \C$ the closed Euclidean disc 
centred at $a$ and of radius $r$.
\begin{defi}
 A real valued function $\varphi:\Oo\rightarrow [-\infty,+\infty)$ is
 said to be 
plurisubharmonic on $\Oo$ if
 \begin{enumerate}
 \item $\varphi$ is upper semi-continuous and $\varphi \not \equiv -\infty$ on $\Oo$.
 \item For all $w\in \Oo$, for all $r>0$ and $\zeta \in \C^n$ such
   that 
$w+\zeta\Delta(0,r)\subset \Oo$,
 $$\varphi(w)\leq \frac{1}{2\pi} \int_0^{2\pi} \varphi(w+\zeta 
 re^{i\theta})\, d\theta.$$
 \end{enumerate}
\end{defi}

We denote by $\PSH(\Oo)$ the set of plurisubharmonic functions on the
domain $\Oo$. 
From the above definition one derives
the following basic properties 
(see \cite[Appendix 1]{LelongGruman} for more details).
\begin{propo} Using the above notations, we have the following.
\begin{enumerate}
\item $\PSH(\Oo)$ is  stable under positive linear combinations.
\item If, $\varphi_1,\varphi_2 \in \PSH(\Oo)$, then $\max\{\varphi_1,\varphi_2\} \in \PSH(\Oo)$.
\item $PSH(\Oo)\subset L^1_{loc}(\Oo)$.  
\item  If $f:\Oo\rightarrow \C$ is a non identically zero holomorphic function, then $\varphi(w)=\log \vert f(w) \vert$ is plurisubharmonic.
\end{enumerate}
\end{propo}

A subset $E\subset \Oo$ is said to be {\it pluripolar} if there exists
a subharmonic function $\varphi:\Oo\rightarrow [-\infty,+\infty)$
such that $E\subset \all{w\in \Oo}{\varphi(w)=-\infty}$. 
From property $(3)$ of the above Proposition it follows that every
pluripolar set is measurable with zero $2n$-dimensional Lebesgue
measure. In the one dimensional case one can show (see  
\cite[p.~57]{Ransford}) that 
every 
Borel\footnote{Non-Borel pluripolar sets do exist, though they are
  still Lebesgue measurable. However in our applications
we will always encounter Borel sets.}
polar set has zero Hausdorff dimension. However, polar sets can be 
uncountable (see \cite[p.~143]{Ransford} for examples of Cantor-like
polar sets).  
One of the key features of plurisubharmonic functions is the following.
\begin{propo}{(Maximum principle)}
Given $\varphi \in \PSH(\Oo)$, either we have for all $w\in \Oo$, 
$$\varphi(w)<\sup_{\zeta \in \Oo} \varphi(\zeta)\,,$$ 
or $\varphi=\sup_{\Oo}\varphi$ is a constant.
\end{propo}
Let $\U\subset \C$ be a domain and let 
$\varphi=\varphi(w,\zeta):\U\times \C\rightarrow [-\infty,+\infty)$ be
a plurisubharmonic function. For all $w\in \U$ we define the order of 
growth $\rho_{\varphi}(w)$ of $\varphi$ (with respect to $\zeta$) 
by 
$$\rho_\varphi(w):=\limsup_{r\rightarrow +\infty} 
\frac{\log (\sup_{\vert \zeta\vert \leq r}
\max\{\varphi(w,\zeta),0\})}{\log r}\,.$$ 
In general, $w\mapsto \rho_\varphi(w)$ is not a subharmonic function 
so the above maximum principle cannot be applied. However
we have the following key result (see \cite[p.~25]{LelongGruman}).
\begin{propo}
\label{cle1} Assume that $\varphi \in \PSH(\U\times \C)$ and that 
$\varphi \geq 1$. Then for all relatively compact domains 
$\U'\subset \U$, there exists a sequence of negative functions 
$\psi_k \in \PSH(\U')$ such that for all $w\in \U'$,
$$\frac{-1}{\rho_\varphi(w)}=\limsup_{k\rightarrow +\infty} \psi_k(w)\,.$$
\end{propo}
To prove our main result we also require the following fact 
(see \cite[p.~25]{LelongGruman}) 
which serves as a substitute for the maximum principle.
\begin{propo}
\label{cle2}
 Let $(\varphi_k)_{k\in \mathbb{N}}$ 
be a sequence in 
$\PSH(\Oo)$, uniformly bounded from above on compact subsets. 
Assume that
 $\limsup_{k\rightarrow +\infty} 
\varphi_k \leq 0$ and that there exists $\xi \in \Oo$ such that 
$\limsup_{k\rightarrow +\infty}\varphi_k(\xi)=0$. Then 
 $$\limsup_{k\rightarrow +\infty} \varphi_k=0\,,$$ 
except on a Borel pluripolar subset of $\Oo$.
\end{propo}

\end{section}

\begin{section}{Complexified homotopies}
\label{sec:holodeform}
In this short section we prove a key lemma which will allow us to
holomorphically deform an arbitrary analytic expanding map
into a suitable finite Blaschke (or anti-Blaschke) product. We start
by the so-called lifting lemma in the analytic category, which  
is a classical result of algebraic topology. 
However since we will need to specify the domains of holomorphy, 
we include a proof for completeness.

\begin{lem}
Let $f:\R \rightarrow \T$ be a real-analytic map. 
Then there exists $\widetilde{f}:\R\rightarrow \R$ real analytic such that for all $x\in \R$
$$f(x)=e^{i \widetilde{f}(x)}.$$
\end{lem}
\begin{proof}
Let $\alpha \in \R$ be such that $e^{i\alpha}=f(0)$.  For all $x \in \R$, set 
$$\widetilde{f}(x):=\frac{1}{i} \int_0^x \frac{f'(t)}{f(t)}\,dt+\alpha.$$
Clearly $\widetilde{f}$ is real analytic and real valued since we have
$$
\Re\left ( \frac{f'(x)}{f(x)} \right)=
\frac{d}{dx}\left( \log \vert f(x) \vert \right)=0\,.
$$
Set $g(x)=e^{i\widetilde{f}(x)}$, then $g$ is a solution of the first order linear ODE
$$g'=\frac{f'}{f}g,$$
and thus is proportional to $f$. Since $g(0)=e^{i\alpha}=f(0)$ we are
done.
\end{proof}

Let $P:\R \rightarrow \T$ denote the universal covering map given by $P(\theta)=e^{i\theta}$ and let $\tau:\T\rightarrow \T$ be an analytic expanding map.
By the above lemma, we can always write
$$\tau\circ P(\theta)=e^{i\widetilde{\tau}(\theta)}$$
for some real analytic map $\widetilde{\tau}:\R\rightarrow \R$. 
Then the quantity (independent of $\theta$)
$$ \mathrm{deg}(\tau):=\frac{\widetilde{\tau}(\theta+2\pi)-\widetilde{\tau}(\theta)}{2\pi}\in \Z,$$
is called the degree (or winding number) of the map and does not depend on the choice of $\widetilde{\tau}$. The goal of this section is to prove the following.
\begin{lem}
\label{lem:analyticdef}
 Let $\tau_0, \tau_1$ be two analytic expanding maps of the circle
 such that 
$\mathrm{deg}(\tau_0)=\mathrm{deg}(\tau_1)$. Then there exist two annuli   
$A_{r_0,R_0}$ and $A_{r_1,R_1}$ with  
$\cl{A_{r_0,R_0}}\subset A_{r_1,R_1}$, 
a complex simply connected neighbourhood 
$\mathcal{U}\supset [0,1]$ and a holomorphic map 
$T:\mathcal{U}\times \cl{A_{r_0,R_0}}\rightarrow \C$
 such that:
 \begin{enumerate}
  \item for all $w\in \mathcal{U}$, $T(w,.)$ satisfies 
  $T(w,A_{r_0,R_0})\supset \cl{A_{r_1,R_1}}$; 
 \item for all $w \in [0,1]$, $T(w,.)$ is an analytic expanding map of $\T$; 
 \item  $T(0,.)=\tau_0$ and $T(1,.)=\tau_1$.
 \end{enumerate}
\end{lem}
\begin{proof}
Consider
$$\widetilde{T}(w,\theta):=e^{i\left \{ (1-w)\widetilde{\tau_0}(\theta)+w \widetilde{\tau_1}(\theta)\right \}},$$
and choose $\epsilon>0$ so small that $\widetilde{T}$ is holomorphic on $\C \times ( \R+i[-\epsilon,+\epsilon])$.
Let $P(\theta)=e^{i\theta}$ and let $A_{r_0,R_0}$ be given by 
$$A_{r_0,R_0}:=P(\R+i[-\epsilon,+\epsilon]).$$
We claim, that there exists a unique map $T:\C\times A_{r_0,R_0}\rightarrow \C$ such that for all $(w,\theta)\in \C \times (\R+i[-\epsilon,+\epsilon])$, we have
$$T(w,P(\theta))=\widetilde{T}(w,\theta).$$
Indeed, since $\mathrm{deg}(\tau_0)=\mathrm{deg}(\tau_1)$, we have (by unique continuation) for all $\theta \in \R+i[-\epsilon,+\epsilon]$,
$$\widetilde{\tau_0}(\theta+2\pi)-\widetilde{\tau_0}(\theta)=\widetilde{\tau_1}(\theta+2\pi)-\widetilde{\tau_1}(\theta).$$ 
Therefore for all $(w,\theta)\in \C \times ( \R+i[-\epsilon,+\epsilon])$, 
$$\widetilde{T}(w,\theta+2\pi)=\widetilde{T}(w,\theta),$$
which guarantees that $T$ is well defined. Since 
$$P:\R+i[-\epsilon,+\epsilon] \rightarrow A_{r_0,R_0}$$
 is a locally biholomorphic map (a holomorphic covering) it follows that
$T$ is a holomorphic map. Remark that by taking $\epsilon$ small enough, we can assume that $T(w,.)$ is holomorphic on a neighbourhood of 
$\cl{A_{r_0,R_0}}$.
We will now restrict the first variable $w$ to a domain of the type 
$$\mathcal{U}:=[0,1]+\Delta(0,\eta),$$ 
where $\eta$ will be taken small, and $\Delta(0,\eta)$ denotes the complex disc centred at $0$ of radius $\eta$. 
To simplify notation further, we set 
$$\widetilde{\tau_w}(\theta):=(1-w)\widetilde{\tau_0}(\theta)+w \widetilde{\tau_1}(\theta).$$
Note that since $\tau_0$ and $\tau_1$ are both
expanding and have same degree, we have 
$$\inf_{(w,\theta)\in [0,1]\times \R} \vert \partial_\theta \widetilde{\tau_w}(\theta) \vert >1.$$
Because $\partial_\theta{\widetilde{\tau_w}(\theta)}$ is real for $w\in [0,1]$ and $\theta$ real, a simple compactness argument shows that it is possible
to choose $\eta>0$ and $\epsilon>0$ so small that 
$$ \rho:=\inf_{(w,\theta)\in \mathcal{U}\times \R+i[-\epsilon,+\epsilon]} \vert \Re\left (\partial_\theta \widetilde{\tau_w}(\theta) \right )\vert >1.$$
Assuming that $\Re(\partial_\theta \widetilde{\tau_w} )>0$,
we now observe that 
$$\log\vert T(w,e^{\epsilon+ib})\vert -\log \vert T(w,e^{ib}) \vert =\int_0^\epsilon \partial_a \left \{ \log \vert T(w, e^{a+ib} )\vert \right \}\,da$$
$$= \int_0^\epsilon \Re( \partial_\theta \widetilde{\tau_w}(b-ia))\,da\geq \rho \epsilon,$$
while 
$$\log\vert T(w,e^{ib})\vert -\log \vert T(w,e^{-\epsilon+ib}) \vert =\int_{-\epsilon}^0 \partial_a \left \{ \log \vert T(w, e^{a+ib} )\vert \right \}\,da\geq \rho \epsilon.$$
Choose $\eta>0$ so small that for all $w\in \mathcal{U}$ and $b \in \R$ we have
$$\log \vert  T(w,e^{ib}) \vert\in [-\overline{\epsilon},+\overline{\epsilon}],$$
where $\overline{\epsilon}=\epsilon \frac{\rho-1}{2}$.
We end up with
$$\vert T(w,e^{\epsilon+ib})\vert \geq e^{\rho \epsilon-\overline{\epsilon}}=e^{\epsilon (\rho+1)/2}:=R>R_0=e^\epsilon,$$
$$\vert T(w,e^{-\epsilon+ib})\vert \leq e^{-\rho \epsilon+\overline{\epsilon}}=e^{-\epsilon (\rho+1)/2}:=r<r_0=e^{-\epsilon}.$$
Therefore, uniformly in $w\in \mathcal{U}$, we have 
\[ T(w,\mathbb{T}_{r_0}) \subset D_{r_1} \text{ and  }
   T(w,\mathbb{T}_{R_0}) \subset D_{R_1}^\infty \] 
% $T(w,.)$ maps $\partial A_{r_0,R_0}$  outside $\cl{A_{r_1,R_1}}$ 
for all 
$$R_0<R_1<R \ \mathrm{and\  all}\  r<r_1<r_0.$$
The proof is similar if
$\Re(\partial_\theta \widetilde{\tau_w} )<0$.
\end{proof}
Note that when $w\not \in [0,1]$, the map
$z\mapsto T(w,z)$ is {\it a priori} no longer preserving the unit
circle. However, it is still a holomorphically expansive map
of some annulus $A_{r_0,R_0} $ in the sense defined previously. 
An explicit homotopy between $z\mapsto z^2$ and a Blaschke
product 
with non trivial spectrum
is provided by 
\begin{equation}
\label{exphom}
T(w,z)=z\left ( \frac{2z-w}{2-wz} \right)\,.
\end{equation}
For a plot of the filled Julia set with 
$w=0.5+0.26i$, for which the invariant set is a quasi-circle, see
Figure~\ref{figJ}. 

\begin{figure}[h]
\centering \includegraphics[scale=0.4]{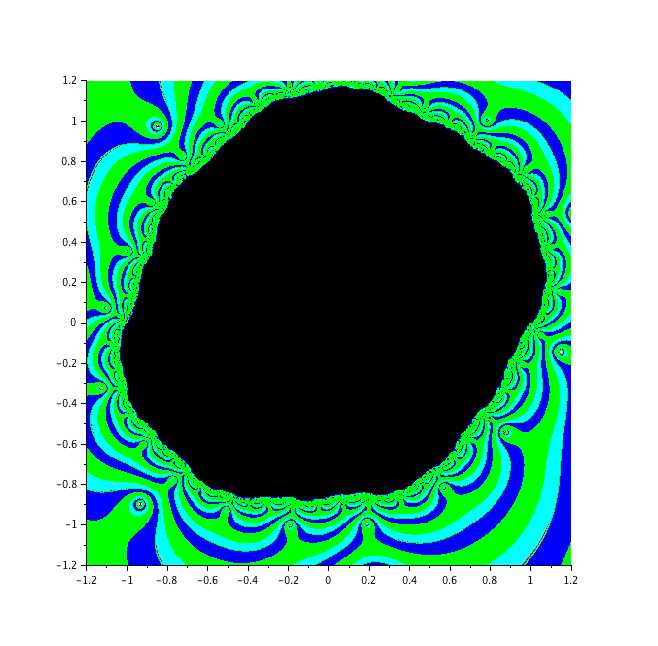} 
\caption{Filled Julia set of the map (\ref{exphom}) with 
$w=0.5+0.26i$.}
\label{figJ}
\end{figure}

\end{section}
\begin{section}{Proof of the main result}
We start with 
$\tau:\T \rightarrow \T$, analytic expanding, 
with degree $d \in \Z$, $\vert d\vert \geq 2$.
We choose a Blaschke 
product $B:\hat{\mathbb{C}}\rightarrow \hat{\mathbb{C}}$ 
with the same degree $d$ and a non-trivial Ruelle eigenvalue
sequence with exponential decay, 
as in Remark~\ref{rem:blexpodecay} and 
Remark ~\ref{rem:ablexpodecay}. Using Lemma
\ref{lem:analyticdef}, we have a holomorphic map 
$$T:\U\times A_{r,R}\rightarrow \C$$ 
such that $T(0,z)=\tau(z)$ and $T(1,z)=B(z)$
with the property that for each $w\in \U$, the map 
$z\mapsto T(w,z)$ is
holomorphically expansive on the annulus $A_{r,R}$. 
By Corollary \ref{co:Cexpoclass}, we know that the dual operator 
$$\lt_{T(w,.)}^\dagger:\monster \rightarrow \monster\,,$$
is a compact trace class operator and we consider the determinant
$$\mathcal{Z}(w,\zeta):=\det(I-e^\zeta \lt_{T(w,.)}^\dagger )\,,$$ 
which defines a holomorphic function on $\U \times \C$. 
Our goal is to investigate the corresponding 
order function (defined for $w\in \U$)
$$\rho(w):=\limsup_{r\rightarrow +\infty} 
\frac{\log (\sup_{\vert \zeta \vert \leq r}
\max\{\log \vert \mathcal{Z}(w,\zeta)\vert,0\})}{\log r}\,.$$
We start with the following simple and useful lemma.
\begin{lem}
 \label{orderest} 
 Let $\mathcal{H}$ be a separable complex Hilbert space and let 
$L:\mathcal{H}\rightarrow \mathcal{H}$ be a trace class operator
 such that the eigenvalue sequence $(\lambda_n(L))_{n\in \mathbb{N}}$ 
satisfies
 $$\vert \lambda_n(L)\vert \leq Ce^{-\alpha n^\beta}\,,$$
 for some $C,\alpha>0$ and $\beta\geq 1$. 
Then, as $\vert \zeta \vert \rightarrow +\infty$, we have
 $$\log\vert \det(I-e^\zeta L)\vert = 
O\left (\vert \zeta  \vert^{\frac{1}{\beta}+1}\right)\,.$$ 
\end{lem}
\begin{proof}
 Write
 $$\log \vert \det(I-e^\zeta L)\vert 
\leq\sum_{n\geq 1} \log(1+Ce^{\vert \zeta\vert} e^{-\alpha n^\beta})$$
 $$\leq N \log(1+Ce^{\vert \zeta\vert})+Ce^{\vert \zeta \vert} 
\sum_{n\geq N+1}  e^{-\alpha n^\beta}\,.$$
 Since 
 $$ \sum_{n\geq N+1}  e^{-\alpha n^\beta}\leq \int_N^{+\infty} e^{-\alpha t^\beta}dt=O(e^{-\alpha N^\beta}),$$
 setting $N=2\lfloor (\vert \zeta  \vert/\alpha)^{1/\beta}\rfloor$ 
now finishes the proof.
 \end{proof}
 The first key observation is the following.
\begin{propo}
 \label{order} Using the above notations, for all $w\in \U$ we have 
$\rho(w)\leq 2$ and $\rho(1)=2$.
\end{propo}
\begin{proof}
 By Corollary \ref{co:Cexpoclass}, we know that 
$\lt_{T(w,.)}^\dagger$ is in the exponential class, 
so it definitely follows from Lemma \ref{orderest} that 
$\rho(w)\leq 2$. On the other hand, for $w=1$, we have by Proposition \ref{product1}  and Proposition \ref{product2}
 the explicit formula (we state it for the orientation preserving case)
 $$\mathcal{Z}(1,\zeta)=
(1-e^\zeta)\prod_{k=1}^\infty (1-e^\zeta\mu^k)(1-e^\zeta \overline{\mu}^k)\,,$$
 where $\mu\in \mathbb{D}\setminus \{0\}$. We shall now show that 
$\rho(1)=2$. Assume to the contrary that $\rho(1)<2$ and fix $\rho$
with $\rho(1)<\rho <2$. 
Now consider the counting function
  $$N(r):=\#\{ \vert \zeta +1\vert \leq r\ :\ \mathcal{Z}(1,\zeta )=0\}\,.$$
Applying Jensen's formula 
(see, for example, \cite[Chapter~4]{Ransford}), 
we have (note that $\mathcal{Z}(1,-1)\neq 0$)
$$\int_0^{2R} \frac{N(r)}{r}dr=\frac{1}{2\pi}\int_{0}^{2\pi} \log\vert
\mathcal{Z}(1,-1+2Re^{i\theta})\vert\,d\theta-\log\vert
\mathcal{Z}(1,-1)\vert
=O(R^\rho)$$
as $R\to \infty$. Observing that 
$$\int_0^{2R} \frac{N(r)}{r}\,dr\geq 
\int_R^{2R} \frac{N(r)}{r}\,dr\geq   
N(R) \int_R^{2R} \frac{1}{r}\,dr
=\log(2)N(R)\,,$$
we obtain as $R\rightarrow +\infty$,
$$N(R)= O(R^{\rho })\,.$$ 
On the other hand, it can be seen from the above explicit 
product formula that zeros of $\mathcal{Z}(1,\zeta)$ 
contain a rank $2$ lattice, which contradicts the above growth
estimate because this implies that 
$N(R)\geq CR^2$ for all $R$ large by a simple lattice counting
argument. 
 %(Gauss circle problem).
\end{proof}

We are now ready to use Proposition \ref{cle1} and \ref{cle2} to 
complete the proof. Fix $(\U_n)$ a compact exhaustion of $\U$.
For all $n$ large enough, one can find by Proposition \ref{cle1} a 
sequence $(\psi_k)_{k\in\mathbb{N}}$ of subharmonic functions on $\U_n$ 
such that
$$\limsup_{k\rightarrow \infty} \psi_k(w)
=\frac{1}{2}-\frac{1}{\rho(w)}\leq 0\,.$$
On the other hand, we know that
$$\limsup_{k\rightarrow \infty} \psi_k(w)=\frac{1}{2}-\frac{1}{\rho(1)}=0\,,$$
which implies by Proposition \ref{cle2} that
$\rho(w)=2$ for all $w\in \U_n\setminus E_n$, 
where $E_n$ is a polar set. Since a countable reunion of polar sets is polar,
we deduce that finally
$\rho(w)=2$ for all $w\in \U\setminus E$ where $E$ is a polar set. 
We know in addition that polar sets have Hausdorff dimension $0$ which is
more than enough to conclude that $E\cap [0,1]$ has Lebesgue measure
$0$. 
Notice that whenever $\rho(w)=2$, we know by Lemma \ref{orderest} 
that the eigenvalue sequence of 
$\lt_{T(w,.)}^\dagger$ has to satisfy
$$\limsup_{n\rightarrow \infty} \vert \lambda_n(\lt_{T(w,.)}^\dagger)
\vert \exp(n^{1+\epsilon}) >0\,, $$
for every $\epsilon >0$, 
otherwise it would produce a contradiction. 
The very end of the proof follows from the observation that for 
all compact subset $K\subset \U$ and all 
$$r<r_1<1<R_1<R,$$
 we have for all $\eta $ small,
$$ \sup_{z\in A_{r_1,R_1}}\vert T(0,z)-T(\eta,z)\vert\leq 
\eta\sup_{K\times A_{r_1,R_1}}  \vert \partial_1 T \vert\,.$$

\end{section}

\section{Acknowledgements}
We are very grateful to Julia Slipantschuk for communicating
Lemma~\ref{Julialem} to us. 
We would also like to 
thank Viviane Baladi for bibliographic assistance. 
FN is supported by ANR ``GeRaSic'' and the 
Institut Universitaire de France.

% \bibliography{bibli1}

\end{document}